\documentclass[11pt]{amsart}
 \usepackage[utf8]{inputenc}
\usepackage{color,amssymb, thm-restate,hyperref}
\usepackage[normalem]{ulem}
\usepackage[shortlabels]{enumitem}

\usepackage[capitalise]{cleveref}
\usepackage{mathrsfs}

\usepackage{ocgx2}
\usepackage{graphicx}

\usepackage[all]{xy}

\usepackage{thmtools}
\usepackage{thm-restate}
\usepackage{todonotes}

\newtheorem{theorem}{Theorem}[section]

\newtheorem{proposition}[theorem]{Proposition}
\newtheorem{corollary}[theorem]{Corollary}

\newtheorem{lemma}[theorem]{Lemma}

\newtheorem{question}[theorem]{Question}

\newtheorem*{claim*}{Claim}
\theoremstyle{definition}

\newtheorem{definition}[theorem]{Definition}
\newtheorem{example}[theorem]{Example}
\newtheorem*{definition*}{Definition}

\newcommand{\filter}{\mathcal{F}}

\newcommand{\Homeo}{\operatorname{Homeo}}
\newcommand{\N}{\mathbb N}

\newcommand{\Q}{\mathbb Q}

\newcommand{\Z}{\mathbb Z}
\newcommand{\F}{\mathcal F}

\newcommand{\A}{\mathcal A}

\newcommand{\w}{\omega}

\newcommand{\im}{\operatorname{im}}

\newcommand{\Sym}{\operatorname{Sym}}
\newcommand{\Stab}{\operatorname{Stab}}

\newcommand{\genset}[1]{\langle #1\rangle}

\newcommand{\makeset}[2]{\left\lbrace #1 \;\middle|\;
 \begin{tabular}{@{}l@{}}
   #2
  \end{tabular}
  \right\rbrace}

\title{Maximal subgroups of homeomorphism groups}
\author{S. Bardyla, L. Elliott, and Y. P\'eresse}

\address{S.~Bardyla: University of Vienna, Institute of Mathematics, Vienna, Austria}
\thanks{The research of the first named author was funded in whole by the Austrian Science Fund FWF [10.55776/ESP399].}
\email{sbardyla@gmail.com}
\address{L. Elliott: University of Manchester, Department of Mathematics and Statistics, UK}
\email{luna.elliott142857@gmail.com}

\address{Y. P\'eresse: University of Hertfordshire, School of Physics, Engineering and Computer Science, UK}
\email{y.peresse@herts.ac.uk}
\makeatletter


\begin{document}

\begin{abstract}
We show that the homeomorphism groups of the following spaces have precisely $2^{2^{\aleph_0}}$ maximal subgroups: the rational numbers~$\Q$, the Baire space $\N^{\N}$, the space $\N\times 2^{\N}$ where $2^{\N}$ is the Cantor set, the ordinal $\omega^2$ under its order topology, and the Sorgenfrey line $\mathbb{S}$. More generally, we find sufficient conditions on a group $G$ acting on a topological space which imply that $G$ has at least $2^{2^{\aleph_0}}$ maximal subgroups.

Moreover, if the groups $\Homeo(\Q)$ and $\Homeo(\N^\N)$ are equipped with the pointwise topology, then it is shown that $\Homeo(\N^\N)$ has precisely $2^{\aleph_0}$ open maximal subgroups, and $\Homeo(\Q)$ has precisely $\aleph_0$ open maximal subgroups and $2^{\aleph_0}$ closed maximal subgroups.
\end{abstract}

\maketitle

\tableofcontents
\section{Introduction}

A proper subgroup $M$ of a group $G$ is called {\em maximal}, if there exists no subgroup $H$ of $G$ such that $M<H<G$. In other words, $M$ is maximal if $M$ together with any element of $G \setminus M$ generates $G$. Maximal subgroups have been studied for a wide range of groups such as countable groups~\cite{GGS}, ample groups~\cite{GV}, Thompson's groups~\cite{A,belk2025type}, branch groups \cite{F}, multi-edge spinal groups~\cite{AKT}, groups of intermediate growth \cite{FG}, Grigorchuck-Gupta-Sidki groups \cite{FT}, $SL_n(\Z)$~\cite{GM}, and finitely generated linear groups~\cite{MS}.

The maximal subgroups of the symmetric group $\Sym(X)$ on a finite set $X$ are classified by the famous O'Nan-Scott Theorem, see \cite{Liebeck_Praeger_Saxl_1988}. Maximal subgroups of $\Sym(\N)$ and related structures have also received a lot of attention, \cite{BST, EMP, HP, KK, LS, MP}. 
It was shown in \cite{subgroups_macpherson_neumann}, that $\Sym(\N)$ has $2^{\mathfrak{c}}$ maximal subgroups where $\mathfrak c$ denotes the cardinality of the continuum. In other words, $\Sym(\N)$ has as many maximal subgroups as it has subsets. This striking result offers an immediate insight into the complexity of the subgroups structure of $\Sym(\N)$ and hence of $\Sym(\N)$ itself. 

If we consider $\N$ to be a discrete topological space, then $\Sym(\N)$ coincides with the homeomorphism group $\Homeo(\N)$. Viewed from this angle, it is natural to ask how many maximal subgroups homeomorphism groups of other topological spaces have. 
In this paper we show that the homeomorphism groups of the following spaces have precisely $2^{\mathfrak{c}}$ maximal subgroups: the discrete space $\N$ (as previously shown in \cite{subgroups_macpherson_neumann}), the rational numbers $\Q$, the Baire space $\N^{\N}$, $\N\times 2^{\N}$ where $2^{\N}$ is the Cantor set, the ordinal $\omega^2$ under its order topology, and the Sorgenfrey line $\mathbb{S}$. More generally, we find sufficient conditions on a group $G$ acting on a topological space which imply that $G$ has at least $2^{\mathfrak c}$ maximal subgroups.

Since $\Homeo(\Q)$ and $\Homeo(\N^\N)$ are particularly important, see~\cite{Bruyns, Droste, Mekler, Neumann, Truss}, we consider their maximal subgroups in more detail.  
We show that, if we endow $\Homeo(\Q)$ and $\Homeo(\N^\N)$ with the pointwise topology, then $\Homeo(\N^\N)$ has precisely $\mathfrak{c}$ open maximal subgroups and $\Homeo(\Q)$ has precisely $\aleph_0$ open maximal subgroups and $\mathfrak c$ closed maximal subgroups.

The paper is structured as follows. Section \ref{section:main_results} contains an overview of our main results, the definitions required to state them precisely, and some of their applications. In Section \ref{section:preliminaries}, we state and prove some general topological lemmas required in subsequent sections. Sections  \ref{section:filters} and \ref{section:partitions} contain the proofs of the main results.

\section{Main results and applications}\label{section:main_results}

 The main result of this paper is Theorem~\ref{bigmaxtheorem} which 
 allows us to count the number of maximal subgroups of a group $G$ acting on a space $X$. We will now define the necessary conditions on $X$ (\cref{def:kappaclast}) and the action of $G$ (\cref{full_emulation}) and give examples of prominent spaces and groups satisfying these conditions. We start with the definitions related to the space $X$.
 
 Recall that a {\em moiety} of a set \(X\) is a subset \(M\) such that \(|M|=|X|=|X\backslash M|\). This inspired the following definition.

\begin{definition}
    A subset \(M\) of a topological space \(X\) is a \emph{topological moiety of} $X$ if \(M\) is clopen and the topological spaces \(M\), \(X\backslash M\), and \(X\) are pairwise homeomorphic.
\end{definition}

Note that $A$ is a topological moiety of $X$ if and only if $X\setminus A$ is a topological moiety of $X$.

\begin{definition}\label{clastic}
    A non-empty topological space \(X\) is called \emph{homeoclastic} if \(X\) satisfies the following conditions:
    \begin{enumerate}
        \item \(X\) is Hausdorff;
        \item the topological moieties of \(X\) form a subbasis for \(X\);
         \item whenever \(\{A,B\}\) is a partition of \(X\) into two clopen sets, then either \(A\) or \(B\) is homeomorphic to \(X\);
        \item a clopen subspace of \(X\) is homeomorphic to \(X\) if and only if it contains a topological moiety of \(X\).
    \end{enumerate}
\end{definition}

Note that condition (4) in Definition \ref{clastic} implies that a homeoclastic space $X$ must be infinite and may be partitioned into $n$ topological moieties for every finite $n$. The next definition is a strengthening of this property. 

\begin{definition}\label{def:kappaclast}
  For an infinite cardinal $\kappa$ a space $X$ is called {\em $\kappa$-homeoclastic} if $X$ is homeoclastic and $X$  can be partitioned into \(\kappa\) topological moieties.
\end{definition}

It is straightforward to verify that a $\kappa$-homeoclastic space is also $\lambda$-homeoclastic for every infinite $\lambda<\kappa$.


\begin{example}\label{example}
    Here are some natural examples of homeoclastic spaces. 
    \begin{enumerate}[\rm(i)]
        \item Every infinite discrete space \(X\) is \(|X|\)-homeoclastic as in this case topological moieties are precisely usual moieties and clopen subsets homeomorphic to \(X\) are simply subsets of cardinality \(|X|\).
        \item The space $\Q$ of rational numbers is \(\aleph_0\)-homeoclastic, as every proper nonempty clopen subspace of \(\Q\) is a topological moiety of \(\Q\), and \(\Q \) is homeomorphic to \(\Q \times \N\).
    \item The Baire space \(\N^\N\) is \(\aleph_0\)-homeoclastic for the same reason;
    \item The Cantor space \(2^{\N}\) is homeoclastic, but not \(\aleph_0\)-homeoclastic, as compact spaces cannot be partitioned into infinitely many topological moieties. 
    \item The topological space \(\N \times 2^{\N}\) is \(\aleph_0\)-homeoclastic by \cref{kappa}. 
  \item Recall that the ordinal \(\omega^2\) is the least ordinal greater than each of the ordinals \(\omega, 2\cdot \omega, 3\cdot \omega,\ldots\). With the order topology this ordinal is homeomorphic to the space \(\N \times \{0,1,\frac{1}{2},\frac{1}{3},\ldots\}\). Taking into account that a clopen subspace of \(\N \times \{0,1,\frac{1}{2},\frac{1}{3},\ldots\}\) is homeomorphic to the entire space precisely if it contains infinitely many of the points \((n,0)\) for \(n\in \N\),  it is routine to check that $\w^2$ equipped with the order topology is \(\aleph_0\)-homeoclastic;
  \item Recall that the Sorgenfrey Line \(\mathbb{S}\) is the real line endowed with the topology whose basis consists of the half open intervals
\([a,b)\), where $a<b$ are arbitrary.
Observe that \(\mathbb{S}\) is homeomorphic to the topological sum of $\aleph_0$-many copies of the subspace \([0,1)\subseteq \mathbb{S}\). It follows from \cite[Theorem~4.3]{Sorg} that every nonempty clopen subspace of $\mathbb{S}$ is homeomorphic to \(\mathbb{S}\). In particular, \([0,1)\) and $\mathbb S\setminus [0,1)$ are homeomorphic to \(\mathbb{S}\).
 At this point, it is clear that \(\mathbb{S}\) is \(\aleph_0\)-homeoclastic. 
 \item \cref{kappa} can be used to construct more \(\kappa\)-homeoclastic spaces for arbitrarily large cardinals \(\kappa\).
    \end{enumerate}
\end{example}

We now move onto the definitions concerning the action of the group $G$. 
Following usual conventions in topology, a group action of a group $G$ on a topological space $X$ is understood to be an action via homeomorphisms on $X$. That is to say, an action of \(G\) on \(X\) is a homomorphism \(\phi_G:G\to \Homeo(X)\).
We say that \(g'\in G\) \emph{acts via \(g\in \Homeo(X)\)} if \(g=(g')\phi_G\).
If  $x\in X$, \(U\subseteq X\) and \(g\in G\), then we write $(x)g$ instead of $(x)((g)\phi_G)$, $Ug$ instead of $U\left((g)\phi_G\right)$, and \(g{\restriction}_U\) instead of \(((g)\phi_G){\restriction}_U\). Unless specified otherwise, the action of a subgroup $G$ of $\Homeo(X)$ is assumed to be the canonical one using the inclusion homomorphism as $\phi_G$.

\begin{definition}\label{defintion:full_action} Let $X$ be a topological space and $G$ be a group acting on \(X\). 
Then $G$ {\em acts fully on $X$} if it satisfies the following property: If \(h\in \Homeo(X)\) is such that there is an open cover \(\mathcal C\) of \(X\) with \(h{\restriction}_U\in \makeset{g{\restriction}_U}{\(g\in G\)}\) for all \(U\in \mathcal C\), then there is \(h'\in G\) acting via the homeomorphism \(h\). 

\end{definition}
Definition \ref{defintion:full_action} above has already appeared in the literature (see~\cite{BEH,Kat}).
The following definition was inspired by the notion of a \emph{vigorous group} from~\cite{BEH}. 
\begin{definition}
Let $X$ be a homeoclastic space and let \(G\) be a group acting on \(X\). Then $G$ \emph{emulates} $\Homeo(X)$ if for any finite partitions \(\mathcal{P}, \mathcal{Q}\) of \(X\) into topological moieties such that $|\mathcal{P}|=|\mathcal{Q}|$ and any bijection \(\phi:\mathcal{P} \to \mathcal{Q}\), there exists \(g_\phi\in G\) such that  \((P)g_\phi = (P)\phi\) for all \(P\in \mathcal{P}\).
\end{definition}

\begin{definition}\label{full_emulation}
    If \(G\) is a group acting on a homeoclastic space \(X\), then we say that \(G\) \emph{fully emulates} \(\operatorname{Homeo}(X)\) if \(G\) emulates \(X\) and acts fully on \(X\).
\end{definition}
Note that the homeomorphism group of any homeoclastic space fully emulates itself.
In the next example, we present other groups which fully emulate $\Homeo(X)$, for some homeoclastic spaces $X$. 
\begin{example}\label{thompsonish}
 We denote the free monoid over a set $A$ by $A^*$. Each of the following groups $G_1$, $G_2$, $G_3$ is a proper subgroup of the group of all homeomorphisms of its associated space that fully emulates the space's homeomorphism group. Note that $G_1$ is isomorphic to Thompson's group \(V\).
    \begin{align*}
   G_1&=\makeset{g\in \Homeo(2^\N)}{for all \(x\in 2^\N\) there is a prefix \(p\) of \(x\) and \(q\in \{0,1\}^*\) \\
     such that for all \(y\in 2^\N\) we have \((py)g=qy\)}\\
     G_2&=\makeset{g\in \Homeo(\N^\N)}{for all \(x\in \N^\N\) there is a prefix \(p\) of \(x\) and \(q\in \N^*\) \\
     such that for all \(y\in \N^\N\) we have \((py)g=qy\)}\\
     G_3&=\makeset{g\in \Homeo(\Q)}{for all \(q\in \Q\) there is a neighborhood \(U\) of \(q\) such \\that \(g{\restriction}_U:U\to (U)g\) is an order isomorphism}
     \end{align*}
\end{example}
We can now state the main result of this paper. 

\begin{restatable}{theorem}{bigmaxtheorem}
\label{bigmaxtheorem}
    Let $X$ be a \(\kappa\)-homeoclastic space and $G$ be a group acting on \(X\). If $G$ fully emulates  $\operatorname{Homeo}(X)$, then \(G\) has at least $2^{2^{\kappa}}$ distinct maximal subgroups.
    Moreover, if $X$ has a basis of size $\kappa$ and the kernel of the action of \(G\) has size at most $2^\kappa$, then \(G\) has precisely $2^{2^{\kappa}}$  maximal subgroups.
\end{restatable}

We obtain the following (almost) immediate corollary to \cref{bigmaxtheorem}.


\begin{corollary}\label{cor:manyexamples}
    The following groups have precisely $2^{\mathfrak{c}}$ many maximal subgroups.
    \begin{enumerate}[\rm(1)]
        \item $\Sym(\N)=\Homeo(\N)$;
        \item $\Homeo(\Q)$;
        \item $\Homeo(\N^{\N})$;
        \item $\Homeo(\N\times 2^\N)$;
        \item $\Homeo(\omega^2)$, where the ordinal $\omega^2$ carries the usual order topology;
        \item $\Homeo(\mathbb{S})$, where $\mathbb{S}$ denotes the Sorgenfrey line.
    \end{enumerate}
    \begin{proof}
In cases (1) to (5), this is an immediate corollary to \cref{bigmaxtheorem} with $\kappa=\aleph_0$. 
The Sorgenfrey line $\mathbb{S}$ does not have a basis of size $\aleph_0$. Nevertheless, Theorem \ref{bigmaxtheorem} still implies that $\Homeo(\mathbb{S})$ has at least $2^{\mathfrak{c}}$ many maximal subgroups. Since $\mathbb{S}$ is Hausdorff, every homeomorphism of $\mathbb{S}$ is determined by its action on the dense subset $\Q$. Thus $|\Homeo(\mathbb{S})| \leq \mathfrak{c}^{|\Q|}=\mathfrak{c}$ and so $\Homeo(\mathbb{S})$ has at most $2^{\mathfrak{c}}$ subsets. In particular $\Homeo(\mathbb{S})$ has no more than $2^{\mathfrak{c}}$ maximal subgroups.
    \end{proof}
\end{corollary}

Since the Cantor space is not \(\aleph_0\)-homeoclastic, the following natural question remains open.
\begin{question}\label{question_cantor}
    How many maximal subgroups does \(\operatorname{Homeo}(2^\N)\) have?
\end{question}
It is routine to check that the stabiliser of any point in \(\operatorname{Homeo}(2^\N)\) is maximal. Thus \(\operatorname{Homeo}(2^\N)\) has at least \(\mathfrak{c}\) maximal subgroups. Item (4) of \cref{cor:manyexamples} brings us close to answering \cref{question_cantor}. In particular, for all \(x \in 2^\N\), the space \(2^\N\backslash \{x\}\) is homeomorphic to \(\N \times 2^{\N}\). 
Since the open neighbourhoods of \(x\in 2^{\N}\) are precisely the sets containing \(x\) with compact complement, every homeomorphism of \(2^\N\backslash \{x\}\) extends to a homeomorphism of \(2^\N\). 
Since the point stabilisers subgroups of $\Homeo(2^\N)$ are isomorphic to $\Homeo(2^\N\times \N)$, \cref{cor:manyexamples} implies that the point stabilisers in \(\operatorname{Homeo}(2^\N)\) are maximal subgroups each of which has \(2^{\mathfrak{c}}\) maximal subgroups. So there is a family of \(2^{\mathfrak{c}}\) subgroups of \(\Homeo(2^\N)\) each of which can generate \(\Homeo(2^\N)\) with two additional elements.

We now consider the case of $\Homeo(\Q)$ in more detail to count how many of its $2^{\mathfrak{c}}$ maximal subgroups are open or closed. The answer to this question obviously depends on which topology we choose to put on $\Homeo(\Q)$. There is no universally accepted canonical topology for $\Homeo(\Q)$. In 
\cite{Rosendal} it was shown  that $\Homeo(\Q)$ has no Polish group topologies and
\cite[Theorem 1.4]{E} implies that the group Zariski topology on \(\operatorname{Homeo}(\Q)\) is not Hausdorff.
In this paper, we first consider the pointwise topology on $\Homeo(\Q)$, which $\Homeo(\Q)$ inherits from $\Sym(\Q)$. It is a natural choice in the context of general permutation groups and has nice topological properties such as metrisability and separability.  

\begin{restatable}{theorem}{BBA}
\label{maximalsubgroup breakdown}
 The topological group $\Homeo(\Q)$ has
   \begin{enumerate}[\rm(1)]
       \item  \(\aleph_0\) open maximal subgroups;
       \item  \(\mathfrak{c}\) closed maximal  subgroups;
       \item  \(2^{\mathfrak c}\) maximal subgroups;
   \end{enumerate}
 where the topology on $\Homeo(\Q)$ is the topology of pointwise convergence generated by the sets $U_{x,y}:=\makeset{f\in\Homeo(\Q)}{ $(x)f=y$}$ over all $x,y \in \Q$.
\end{restatable}

The main potential argument against choosing the pointwise topology for $\Homeo(\Q)$ is that it ``ignores'' the usual topology on $\Q$. In particular, the \emph{evaluation map} \((q,h)\mapsto (q)h\) from \(\Q\times \Homeo(\Q)\) to \(\Q\) is not continuous under the pointwise topology on $\Homeo(\Q)$ and the natural topology on $\Q$.
Having a continuous evaluation map would clearly be desirable. However, there is a trade-off. In \cite{Concilio}, it is shown that the least Hausdorff group topology on \(\Homeo(\Q)\) such that the evaluation map is continuous contains the closed-open topology, which is generated by the sets $$U_{A,B}:=\{h\in \Homeo(\Q): (A)h\subseteq B\},$$ where $A$ is a closed subset of $\Q$ and $B$ is an open subset of $\Q$. For each $r\in \mathbb R\setminus \Q$ let $$A_r^{-}=\makeset{x\in \Q}{$x<r$}\qquad \hbox{ and } \qquad A_r^{+}=\makeset{x\in \Q}{$x>r$}.$$
For irrational numbers $a,r$ define the set \[W_{a,r}= U_{A^{-}_{a},A^{-}_r}\cap U_{A^{+}_{a},A^{+}_r}= 
\makeset{f\in \Homeo(\Q)}{\((A^-_{a}) f=A^-_{r}\)}.\]
Then for a fixed $a\in\mathbb R\setminus\mathbb Q$, the family $\makeset{W_{a,r}}{$r\in\mathbb{R}\setminus\mathbb{Q}$}$ consists of pairwise disjoint sets that are open with respect to the closed-open topology. 
It follows that group topologies on \(\Homeo(\Q)\) under which the evaluation map is continuous are not separable. Nevertheless, we consider such topologies and obtain the following result. 
\begin{restatable}{theorem}{BBB}
\label{maximal-clopen-open}
Let \(\operatorname{Homeo}(\Q)\) be equipped with a group topology such that the evaluation map \((h,q)\to (q)h\) from \((\Q\times \operatorname{Homeo}(\Q) )\to \Q\) is continuous when $\Q$ is given its usual topology.
Then \(\operatorname{Homeo}(\Q)\) has 
\begin{enumerate}[\rm(1)]
\item at least \(\mathfrak{c}\) open maximal subgroups;
\item \(2^{\mathfrak{c}}\) closed maximal subgroups. 
\end{enumerate}
\end{restatable}

\section{Topological preliminaries}\label{section:preliminaries}

The proof of the following two lemmas are straightforward, and so are left for the reader.
\begin{lemma}\label{trans}
Let $A$ be a topological moiety of a space $X$ and $B$ be a topological moiety of $A$. Then $B$ is a topological moiety of $X$.    
\end{lemma}

\begin{lemma}\label{transactonmoieties}
Let $A$ and \(B\) be topological moieties of a space $X$. If \(G\) is a group acting on \(X\) which emulates \(\Homeo(X)\), then there is \(g\in G\) with \((A)g=B\). 
\end{lemma}

The $2^{\mathfrak c}$ many maximal subgroups of $\Sym(\N)$ given in \cite{subgroups_macpherson_neumann} are the stabilisers of ultrafilters on $X$. In this paper, we similarly construct maximal subgroups of homeomorphism groups using suitable filters on the underlying spaces. Recall that a collection $\mathcal{F}$ of subsets of $X$ is a \emph{filter} on $X$ if  $\emptyset \notin \mathcal{F}$, and $\F$ is closed under finite intersections and taking supersets. A subset $\mathcal B$ of a filter $\F$ is a {\em subbase} for $\F$ if for each $F\in \F$ there exists a finite family $\mathcal C\subseteq \mathcal B$ such that $\bigcap \mathcal C\subseteq F$. A filter $\mathcal{F}$ on $X$ is an \emph{ultrafilter} if for every $A \subseteq X$ either $A \in \mathcal{F}$ or $X \setminus A \in \mathcal{F}$. 

A filter $\F$ on a space $X$ is called
\begin{itemize}
\item {\em a topological moiety filter}, abbreviated as {\em TM filter}, if $\F$ possesses a subbase consisting of topological moieties of $X$;
\item {\em a topological moiety ultrafilter}, abbreviated as {\em TM ultrafilter}, if $\F$ is a topological moiety filter, and for every topological moiety $A$ of $X$ either $A\in \F$ or $X\setminus A\in\F$.
\end{itemize}

Recall that a filter on a set $X$ is an ultrafilter if and only if it is maximal among the set of all filters on $X$ with respect to containment. We now show that the analogue holds for TM filters. 

\begin{lemma}\label{lem:maximalultrafilters}
    A TM filter on a topological space \(X\) is a TM ultrafilter if and only if it is maximal among the TM filters on \(X\).
\end{lemma}
\begin{proof}
\((\Leftarrow):\) Let \(\mathcal{F}\) be a maximal TM filter, and $\mathcal B$ be a subbase of $\F$ consisting of topological moieties.
Suppose for a contradiction that there is a partition \(\{A,B\}\) of \(X\) into topological moieties such that neither \(A\) nor \(B\) belongs to \(\mathcal{F}\). Since $B\notin \F$, we get that \(\mathcal{F}\) contains no subset of \(B\). 
It follows that every element of \(\mathcal{F}\) intersects \(A\).
So \[\mathcal{F}'=\makeset{S}{there is $F\in \mathcal{F}$ with \(F\cap A \subseteq S\)}\] is a filter with subbase \(\mathcal{B}\cup \{A\}\). As \(A\) was a topological moiety, it follows that \(\mathcal{F}'\) is a TM filter strictly containing \(\mathcal{F}\), we have reached a contradiction.

\((\Rightarrow):\) Let \(\mathcal{F}\) be a TM ultrafilter and \(\mathcal{F}'\) be a TM filter containing \(\mathcal{F}\). 
If \(\mathcal{F}\) contains all the topological moieties in \(\mathcal{F}'\), then it would follow \(\mathcal{F}=\mathcal{F}'\).
Suppose for a contradiction that \(U\in \mathcal{F}' \backslash \mathcal{F}\) is a topological moiety. Since $\F$ is a TM ultrafilter, either \(U\in \mathcal{F}\) or \(X\setminus U\in \mathcal{F}\). Thus \(X\backslash U\in \mathcal{F}\). As \(\mathcal{F} \subseteq \mathcal{F}'\), it follows that both \(U\) and its complement belong to \(\mathcal{F}'\). Then \(\varnothing=U\cap (X\setminus U)\in \mathcal{F}'\), a contradiction.
\end{proof}

It is a well known result that every infinite set $X$ has $2^{2^{|X|}}$ many ultrafilters, see Theorem~3.6.11 from \cite{Eng}. The next lemma is a generalisation of this fact, if we view an infinite set $X$ as a discrete space.

\begin{lemma}\label{2chomeoclast}
Each \(\kappa\)-homeoclasitc space \(X\) possesses at least  $2^{2^{\kappa}}$ distinct TM ultrafilters.   
\end{lemma}

\begin{proof}

Let $\kappa$ be an infinite cardinal and $X$ be a  \(\kappa\)-homeoclasitc space.  By definition, there exists a partition \(\A=\makeset{A_\alpha}{$\alpha<\kappa$}\) of $X$ into topological moieties . 
Fix a selector \(\makeset{a_\alpha}{$\alpha \in\kappa$}\), where $a_\alpha\in A_\alpha$. 
For each ultrafilter $u$ on the set $\kappa$, let $\F_u$ be the filter on $X$ that is generated by the subbase $$\{W\subseteq X: W \hbox{ is a topological moiety and } \{\alpha\in \kappa: a_\alpha\in W\}\in u\}.$$ 
Let us show that $\F_u$ is a TM ultrafilter. This filter is a TM filter by definition. Fix any topological moiety $A$ of $X$. If $T=\{\xi\in\kappa: a_\xi\in A\}\in u$, then $A\in\F_u$. If $T\notin u$, then $\{\xi\in\kappa: a_\xi\in X\setminus A\}=\kappa\setminus T\in u$. In the latter case, $X\setminus A\in\F_u$. Hence $\F_u$ is a TM ultrafilter. 

Fix distinct ultrafilters $u,v$ on $\kappa$. There exist $U\in u$ and $V\in v$ such that $U\cap V=\emptyset$. Note that $$\bigcup_{n\in U}A_n\in \F_u, \qquad \bigcup_{n\in V}A_n\in \F_v,\qquad \hbox{and}\qquad \left(\bigcup_{n\in U}A_n \right)\cap\left(\bigcup_{n\in V}A_n\right)=\emptyset.$$ It follows that $\F_u\neq \F_v$. 
 Thus there are at least as many TM ultrafilters on \(X\) as there are ultrafilters on \(\kappa\).
Theorem~3.6.11 from \cite{Eng} implies that there are \(2^{2^\kappa}\) ultrafilters on $\kappa$.
\end{proof}

The following proposition can be used to construct $\kappa$-homeoclastic spaces
for arbitrarily large cardinals $\kappa$, as mentioned in \cref{example}.
\begin{proposition}\label{kappa}
    Suppose that \(\kappa\) is an infinite cardinal with the discrete topology and \(X\) is an infinite Hausdorff topological space with a basis of clopen sets such that every nonempty clopen subspace of \(X\) is homeomorphic to \(X\).
    If for all cardinals \(\gamma <\kappa\) with the discrete topology we have that \(\gamma \times X \) is not homeomorphic to \(\kappa \times X\), then \(\kappa \times X\) is a \(\kappa\)-homeoclastic space.
\end{proposition}
\begin{proof}
    First we check the four conditions from Definition~\ref{clastic}.
    
    (1) The product space $\kappa \times X$ is Hausdorff since the discrete space $\kappa$ and $X$ are both Hausdorff.

    (2) Let \(B\) be the set of topological moieties of \(X\). Since each non-empty clopen subset of $X$ is homeomorphic to $X$, it follows that $B$ coincides with the family of all proper non-empty clopen subsets of $X$. Since $X$ has a basis of clopen sets, the sets \(\{\alpha\}\times U\) where \(\alpha\in \kappa\) and \(U\in B\) form a basis for \(\kappa \times X\).
    It follows that the family $$S=\makeset{\bigcup_{\alpha\in M}(\{\alpha\}\times U_\alpha)}{\(M \hbox{ is a moiety of }\kappa \hbox{ and } U_\alpha\in B \hbox{ for all }\alpha\in M\)}$$ forms a subbase of $\kappa \times X$. Notice that $S$ consists of topological moieties of $\kappa\times X$. 

    (3) Suppose that \(\{A,B\}\) is a partition of \(\kappa \times X \) into clopen sets. Without loss of generality, suppose that \(I:=\makeset{\alpha\in \kappa}{\((\{\alpha\}\times X)\cap A \neq \emptyset\)}\) has cardinality \(\kappa\). For each \(i\in I\), we have that \((\{i\}\times X)\cap A\) is homeomorphic to a non-empty clopen subspace of \(X\). By the assumption, each of the spaces \((\{i\}\times X)\cap A\)  for \(i\in I\) is homeomorphic to \(X\). As \(A\) is the topological sum of the spaces \((\{i\}\times X)\cap A\) for \(i\in I\) and $|I|=\kappa$, it follows that \(A\) is homeomorphic to \(\kappa \times X\).

(4) Suppose that \(U\) is a clopen subset of \(\kappa \times X\). If \(U\) contains a topological moiety $T$, then $T$ is homeomorphic to $\kappa \times X$ and $V=U \setminus T$ is clopen in $\kappa \times X$. Hence $V$ is a union $V=\bigcup_{\alpha \in \kappa} \{\alpha\} \times V_{\alpha}$ where each $V_{\alpha}$ is clopen (possibly empty) in $X$. If we use $\sqcup$ to denote topological sums, then 
    \[U=T \sqcup V\cong (\sqcup_{\alpha\in \kappa} X)\sqcup (\sqcup_{\alpha\in \kappa} V_{\alpha}) \cong \sqcup_{\alpha\in \kappa} (X\sqcup V_{\alpha})\cong \sqcup_{\alpha\in \kappa} X\cong \kappa \times X.\]
    
    Suppose now that \(U\) is homeomorphic to \(\kappa\times X\). We need to show that \(U\) contains a topological moiety of \(\kappa\times X\).
   For each $\xi\in\kappa$ the set \(U_{\xi}:=U\cap (\{\xi\}\times X)\) is either empty or homeomorphic to \(X\) since it is clopen in \(\{\xi\}\times X \cong X\). 
   Let \(I_U:=\makeset{\xi\in \kappa}{\(U_\xi\neq \emptyset\)}\). 
   Note that \(U\) is homeomorphic to \(I_U \times X\).
   So if \(|I_U|< \kappa\) then, by the assumption, \(U\) is not homeomorphic to \(\kappa \times X\), a contradiction. Thus \(|I_U|=\kappa\). Let \(M\) be a moiety of \(I_U\). 
   Then \(\bigcup_{\xi\in M} U_\xi\) is a topological moiety of \(\kappa \times X\) contained in \(U\), as required.

    To verify that the condition of \cref{def:kappaclast} holds, let \(\makeset{M_\alpha}{$\alpha< \kappa$}\) be a partition of \(\kappa\) into moieties. Then \(\makeset{M_\alpha \times X}{$\alpha< \kappa$}\) is a partition of \(\kappa\times X\) into topological moieties. Hence \(\kappa \times X\) is a \(\kappa\)-homeoclastic space.
\end{proof}

\section{Maximal subgroups via topological moiety ultrafilters}\label{section:filters}


The aim of this section is to prove \cref{bigmaxtheorem}.

\begin{definition}\label{def:supdef}
    Let \(G\) be a group acting on a topological space \(X\), $f\in G$, and \(A\subseteq X\). Then the {\em support} of $f$ is the set of all $x\in X$ such that $(x)f\neq x$, and 
    $G_A$ denotes the subgroup of $G$ consisting of elements whose support is contained in \(A\). 
\end{definition}
\begin{proposition}\label{marvelous-stabilisers}
    If \(X\) is a \(\kappa\)-homeoclastic space, \(U\subseteq X\) is a topological moiety and \(G\) fully emulates \(\Homeo(X)\), then \(G_U\) fully emulates \(\Homeo(U)\) (via the natural action of \(G_U\) on \(U\)).
\end{proposition}
    \begin{proof}
        Let \(P\) and \(Q\) be finite partitions of \(U\) into clopen sets and \(\phi:P\to Q\) be a bijection.
        Let \(P':=P\cup \{X\backslash U\}\) and \(Q':=Q\cup \{X\backslash U\}\). Let \(\phi':=\phi \cup \{(X\backslash U,X\backslash U)\} \), i.e., $(x)\phi'=(x)\phi$ for all $x\in P$ and $(X\backslash U)\phi'=X\backslash U$.
        Note that \(\phi':P'\to Q'\) is a bijection.
As \(G\) emulates \(\Homeo(X)\), there is \(g\in G\) such that for all \(p\in P'\), we have \((p)g=(p)\phi'\).
As \(G\) acts fully on \(X\), there is \(k\in G\) acting via the homeomorphism \(h\) of \(X\) extending \(g{\restriction}_U\) and the identity map on \(X\backslash U\). It is clear that \(k\in G_U\), and $(p)k=(p)\phi$ for all $p\in P$. Hence $G_U$ emulates $\Homeo(U)$.

Let \(g\in \Homeo(U)\) and suppose that \(C\) is an open cover of \(U\) such that for all \(V\in C\) there is \(h_V\in G_U\) with \(h_V{\restriction}_V=g{\restriction}_V\).
We must show that there is \(k\in G_U\) which acts on \(U\) via the homeomorphism \(g\).
Note that \(C\cup \{X\backslash U\}\) is an open cover of \(X\). 
Let \(h_{X\backslash U}\) be the identity of \(G\) and
let \(k\in \Homeo(X)\) be the element which fixes \(X\backslash U\) pointwise and such that \(k{\restriction}_U=g\).
As \(G\) acts fully on \(X\) and \(k{\restriction}_{V}=h_V{\restriction}_{V}\) for all \(V\in C\cup \{X\backslash U\} \), it follows that there is \(k'\in G\) such that \(k'\) acts on \(X\) via the homeomorphism \(k\). 
Note that \(k'\in G_U\) acts on \(U\) via the homeomorphism \(g\), as required.
\end{proof}

Results like the one discussed in the following lemma are very common in the study of homeomorphism groups.
See for example Property E from \cite{GV}.

\begin{lemma}\label{smallsuppgen}
Let $X$ be a homeoclastic space and $G$ be a group fully emulating $\Homeo(X)$. 
Suppose that \(\{A,B,C\}\) is a partition of \(X\) into topological moieties.
Then 
$$\genset{G_{A\cup B}, G_{B\cup C}}=
G.$$  
\end{lemma}

\begin{proof}
    For convenience, denote the subgroup $\genset{G_{A\cup B}, G_{B\cup C}}$ of $G$ by $H$.
    Let $g\in G$ be arbitrary. 
    We will show that there exists $h\in H$ such that $(A\cup B)gh=A\cup B$. 
    As \(G\) acts fully on \(X\), this implies that there is an element $y$ of \(G_{A\cup B}\) which acts on \(A\cup B\) via the homeomorphism \((gh){\restriction}_{A\cup B}\) and fixes \(C\) pointwise. Similarly, there is an element $z$ of \(G_{C}\) which acts on \(C\) via the homeomorphism \((gh){\restriction}_{C}\) and fixes \(A\cup B\) pointwise.
    Thus the element \(yz\in G_{A\cup B}G_{C}\) acts on \(X\) via the same homeomorphism as \(gh\). As \(G_{A\cup B}\) contains the kernel of the action homomorphism \(\phi_G\), this implies that  $gh\in G_{A\cup B}G_C$ and so \(g\in G_{A\cup B}G_Ch^{-1}\subseteq H\) as required.
    
    We must now find $h\in H$ such that $(A\cup B)gh=A\cup B$.  Let \(\{B_1,B_2\}\) be a partition of \(B\) into topological moieties of $X$. Since $C$ is a topological moiety, $A\cup B=X\setminus C$ is also a topological moiety. Since $g$ induces a homeomorphism of $X$, the sets \((A\cup B)g\) and \((C)g\) are topological moieties of \(X\).  Note that 
    $$(C)g = ((C)g \cap (A\cup B_1)) \cup ((C)g \cap (C\cup B_2)).$$
    As \((C)g\cong X\) is homeoclastic, condition (3) of \cref{clastic} implies that one of the sets \((C)g \cap (A\cup B_1)\) or \((C)g \cap (C\cup B_2)\) is homeomorphic to \((C)g\). Condition (4) of \cref{clastic} then implies that this set contains a topological moiety of \((C)g\).
    We first show that in both cases, there is \(h\in H\) such that \((A\cup B)gh \subseteq B\).
    
    Case 1: First consider the case when \((C)g \cap (A\cup B_1)\) contains a topological moiety \(D\) of \((C)g\). 
    Note that \(D\cong  (C)g \cong X\cong A\cup B_1\) are homeoclastic. Thus, condition (4) of \cref{clastic} applied to $A\cup B_1$ and its clopen subspace $D$ implies that \(D\) contains a topological moiety \(E\) of \(A\cup B_1\). 
    Let \(F\) be a topological moiety of \(E\).
    We now have \(F\subseteq E\subseteq D\subseteq (C)g \cap (A\cup B_1)\). 
    Note that \(A\cup B_1\) is homeomorphic to \(X\), and so are \(B_2\) and \(A\cup B\).
    Thus \(A\cup B_1\) is a topological moiety of \(A\cup B\).
    We now have that \(F\) is a topological moiety of \(E\), which is a topological moiety of \(A\cup B_1\), which is a topological moiety of \(A\cup B\).
    Thus, by \cref{trans}, \(F\) is a topological moiety of \(A\cup B\), and \(E\backslash F\) is a topological moiety of $X$.
    As \(X\backslash F\) and \(X\backslash E\) are clopen and homeomorphic to \(X\), it follows that \(E\backslash F\) is a topological moiety of \(X\backslash F\).
    
    By \cref{marvelous-stabilisers}, \(G_{A\cup B}\) emulates \(\Homeo(A\cup B)\).
    So, by \cref{transactonmoieties}, there is \(h_1\in G_{A\cup B}\) with \((F)h_1=A\). This implies that \((E\backslash F)h_1\) is a topological moiety of \(B\cup C\). By applying \cref{marvelous-stabilisers} and \cref{transactonmoieties} again, we can find \(h_2\in G_{B\cup C}\) with \((E)h_1h_2=A\cup C\). Put $h=h_1h_2\in H$. Then \(((C)g)h\supseteq A\cup C\), which implies that \((A\cup B)gh \subseteq B\).

    Case 2: Symmetric with Case 1.
    
We now have \(h\in H\) with \((A\cup B)gh \subseteq B\). Using \cref{marvelous-stabilisers} and \cref{transactonmoieties}, let \(h_3\in H\) be such that $(B)h_3=B_1$. Then \((A\cup B)ghh_3 \subseteq B_1\) is a clopen set homeomorphic to \(B\) which is disjoint from  \(B_2\). As \(B\) is homeoclastic, condition (4) of \cref{clastic} implies that \(B\backslash ((A\cup B)ghh_3)\) is homeomorphic to \(B\). Since $g,h,h_3$ induce homeomorphisms and $A\cup B$ is homeomorphic to $B$, we get that \((A\cup B)ghh_3\) is homeomorphic to $B$. Thus \((A\cup B)ghh_3\) is a topological moiety of \(B\).
Since $B$ is a topological moiety of $A\cup B$, \cref{trans} implies that \((A\cup B)ghh_3\) is a topological moiety of \(A \cup B\). Since $A\cup B_1$ is a topological moiety of $A\cup B$ (recall case 1), and $A\cup B$ is a topological moiety of $X$, it follows from \cref{trans} that $A\cup B_1$ is a topological moiety of $X$.
By \cref{marvelous-stabilisers} and \cref{transactonmoieties}, there is \(h_4\in G_{A\cup B}\leq H\) with \(((A\cup B)ghh_3)h_4=A\cup B_1\). By \cref{marvelous-stabilisers} and \cref{transactonmoieties}, there is \(h_5\in G_{B\cup C}\) such that \((B_1)h_5=B\). It follows that \((A\cup B)ghh_3h_4h_5=A\cup B\) as required.  
\end{proof}

\begin{lemma}\label{lem:filtact}
    Suppose that \(G\) is a group acting on a topological space \(X\).
    Let \(\mathbb {TMU}\) denote the set of TM ultrafilters on $X$.
    There is a well-defined action of \(G\) on \(\mathbb {TMU}\) (we view $\mathbb {TMU}$ as discrete) given by
    \[(\mathcal{F})g=\makeset{(U)g}{\(U\in \mathcal{F}\)}\]
    for all \(\mathcal{F}\in \mathbb{TMU}\) and \(g\in G\). In particular, the stabiliser 
    $G_{\filter}$ of any \(\mathcal{F}\in \mathbb{TMU}\) is a subgroup of $G$.
\end{lemma}
\begin{proof}
 As each \(h\in G\) induces a homeomorphism of \(X\), it follows that the image of a topological moiety is a topological moiety.
 Thus, for each TM filter \(\mathcal{F}\), the set \((\mathcal{F})h=\makeset{(U)h}{\(U\in \mathcal{F}\)}\) is a TM filter.
 Thus each element of \(G\) induces a permutation of the set of TM filters. It is routine to check, using \cref{lem:maximalultrafilters}, that the image $(\F)h$ of a TM ultrafilter $\F$ is also a TM ultrafilter.
\end{proof}

\begin{definition}
Let $G$ be a group acting on a topological space \(X\) and $\F$ be a TM ultrafilter. By $G_\F$ we denote the stabiliser of \(\filter\), using the action from \cref{lem:filtact}.   
\end{definition}

\begin{theorem}\label{theorem-stabilisers-clopen-maximal}
    Let $X$ be a homeoclastic space and $\filter$ be a TM ultrafilter on $X$. If \(G\) is a group acting on \(X\) which fully emulates \(\Homeo(X)\), 
    then  $G_{\filter}$ is a maximal subgroup of $G$. 
    Moreover, distinct TM ultrafilters give rise to distinct stabilisers. 
\end{theorem}
\begin{proof}
 Fix a TM ultrafilter $\filter$ on $X$. By \cref{lem:filtact}, $G_\F$ is a subgroup of $G$.
First note that \(G_\filter\) does not emulate \(\Homeo(X)\) as if \(U\in \filter\) is a topological moiety then there is no \(g\in G_\filter\) with \((U)g=X\backslash U\). Hence \(G_\filter \neq G\). 

Since the set of topological moieties from $\F$ forms a subbase of $\F$ and $h$ preserves topological moieties of $X$, we obtain that \[(\makeset{U\in \mathcal{F}}{\(U\) is a topological moiety})h\] is a subbase for the TM ultrafilter \((\mathcal{F})h\). 
 \cref{lem:maximalultrafilters} implies the following:
\begin{equation}\label{moiety-filter-equivalence}
h\in G_{\filter} \iff (\makeset{U\in \mathcal{F}}{\(U\) is a topological moiety})h \subseteq \mathcal{F}.   
\end{equation}

To show that $G_{\filter}$ is maximal, let $h\in G\setminus G_{\filter}$ be arbitrary.
Then Equation~\ref{moiety-filter-equivalence} above implies that there exists a topological moiety $U\in \filter$ such that $(U)h\not \in \filter$. Note that \(G_{X\backslash U}\leq G_\filter\), thus \[G_{X\backslash (U)h}=h^{-1}G_{X\backslash U}h \leq 
 \langle G_\filter ,h\rangle.\] 
 As \(\filter\) is a TM ultrafilter and $(U)h\not \in \filter$, we have \(X\backslash (U)h\in\F\). Note that \(X\backslash (U)h\) is a topological moiety. Since $X$ is homeoclastic and \(X\backslash (U)h\cong X\), we can partition the subspace \(X\backslash (U)h\) of \(X\) into two topological moieties \(A\) and \(B\).
 We now have that \(\{A,B,(U)h\}\) is a partition of \(X\) into topological moieties, \(A\cup B \in \F\), and \(G_{A\cup B}=G_{X\backslash (U)h}\leq 
 \langle G_\filter ,h\rangle.\) As \(\F\) is a TM ultrafilter, it follows that exactly one of \(A\) and \(B\) belongs to \(\F\). Assume without loss of generality that \(A\in \F\). Thus \(G_{B\cup (U)h} \leq G_{\F}\).
 
We now have
\[\langle G_{A\cup B},G_{B\cup (U)h} \rangle \leq 
 \langle G_\filter ,h\rangle.\]
 \cref{smallsuppgen} then implies that 
 \[G= \langle G_{\F},h \rangle.\]
 This completes the proof that the group \(G_{\F}\) is maximal. Suppose that $\F$ and $\F'$ are distinct TM ultrafilters.
It remains to check that $G_\F\neq G_{\F'}$. Since $\mathcal{F}\neq \mathcal{F}'$, there exists a topological moiety \(U \in \mathcal{F}\backslash \mathcal{F}'\).
Since $\mathcal{F}'$ is a TM ultrafilter, $X \setminus U \in  \mathcal{F}'$ and so $G_{U}\leq G_{\mathcal{F}'}$. Let $\{A, B\}$ be a partition of $U$ into topological moieties and assume without loss of generality that $A \in \mathcal{F}$. Since $G$ fully emulates $\Homeo(X)$, \cref{transactonmoieties} implies that there exists $g \in G_{U}\leq G_{\mathcal{F}'}$ such that $(A)g=B$ and $(B)g=A$. Then $g \in G_{\mathcal{F}'}\setminus G_{\mathcal{F}}$. In particular, $G_{\mathcal{F}'}\neq G_{\mathcal{F}}$, as required. 
\end{proof}

We may now prove the main result of this paper, which we restate for convenience. 

\bigmaxtheorem*
\begin{proof}
    By Lemma~\ref{2chomeoclast}, there are at least $2^{ 2^{\kappa}}$ distinct TM ultrafilters on $X$. Thus the first part of the result follows directly from Theorem \ref{theorem-stabilisers-clopen-maximal}. If $X$ has a basis of size $\kappa$, then \(|X|\leq 2^{\kappa}\). Moreover, \(X\) has a dense subset \(D\) of size \(\kappa\), so if $\phi_G$ denotes the homomorphism defining the action of $G$, then \[|G|=|\ker(\phi_G)|\cdot|\im(\phi_G)|\leq |\ker(\phi_G)|\cdot |X^D|\leq  2^{\kappa}\cdot (2^\kappa)^{\kappa}=2^{\kappa}.\] Thus \(G\) contains at most \(2^{2^{\kappa}}\) subsets. In particular, $G$ has at most $2^{2^{\kappa}}$ many (maximal) subgroups.
\end{proof}


\section{Maximal subgroups via finite partitions}\label{section:partitions}

Recall that the third item of \cref{maximalsubgroup breakdown} follows from \cref{cor:manyexamples}.
In this section, we prove the remaining two items of \cref{maximalsubgroup breakdown} as well as \cref{maximal-clopen-open}.

\begin{theorem}\label{closedstab}
    Let $n\geq 2$ and let $P=\{\Sigma_0, \dots, \Sigma_{n-1}\}$ be a partition of the rational numbers $\Q$ into clopen sets. Then the stabiliser 
    $$\Stab(P)=\makeset{f\in \Homeo(\Q)}{\((\Sigma)f\in P \text{ for all }\Sigma \in P\)}$$
    of $P$ is a closed maximal subgroup of $\Homeo(\Q)$ with respect to the pointwise topology.
\end{theorem}
\begin{proof}
It is straightforward to check that $\Stab(P)$ is a closed subgroup of $\Homeo(\Q)$. 
    Let $h \in \Homeo(\Q)\setminus \Stab(P)$ be arbitrary. There exists $\Sigma \in P$ such that $(\Sigma)h \not \in P$. So, either $(\Sigma)h$ has non-empty intersection with at least two elements of $P$, or $(\Sigma)h$ is a proper subset of some element of $P$. In the latter case, since $P$ is finite, there must exist some $\Sigma' \in P$ such that $(\Sigma')h$ has non-empty intersection with at least two distinct elements of $P$. So assume without loss of generality that $(\Sigma)h\cap \Sigma_0$ and $(\Sigma)h\cap \Sigma_1$ are non-empty.

    Let $Y=(\Sigma)h\cap (\Sigma_0 \cup \Sigma_1)$. Recalling the notation from \cref{def:supdef}, note that if $g \in \Homeo(\Q)_Y$, then $t=hgh^{-1}\in \Homeo(\Q)_\Sigma\leq \Stab(P)$.
    Hence $g=h^{-1}th\in \genset{\Stab(P), h}$ and so $\Homeo(\Q)_Y\leq \genset{\Stab(P), h}$. It follows from \cref{marvelous-stabilisers} and \cref{smallsuppgen} applied three times that 
    \[\Homeo(\Q)_{\Sigma_0\cup \Sigma_1}=\langle \Homeo(\Q)_{\Sigma_0\cup Y}, \Homeo(\Q)_{\Sigma_1\cup Y}\rangle \leq \genset{\Stab(P), h}.\] 
    By conjugating with elements of $\Stab(P)$, we conclude that $\Homeo(\Q)_{\Sigma_i\cup \Sigma_j}\leq \genset{\Stab(P), h}$ for all $i,j\in \{0,1, \dots, n-1\}$. By now repeatedly applying \cref{smallsuppgen} using the sets \(\bigcup_{i<j}\Sigma_i\) and \(\Sigma_{j-1}\cup \Sigma_j\) for each \(j\in \{1,2,\ldots,n-1\}\) we conclude that $\genset{\Stab(P), h}=\Homeo(\Q)_{\bigcup_{i=0}^n \Sigma_i}=\Homeo(\Q)$, as required. 
\end{proof}

\begin{definition}
    We say that a subgroup \(G\leq \Sym(X)\) is \emph{\(n\)-transitive} if for all $n$-tuples \((x_0,\ldots, x_{n-1})\) and \((y_0,\ldots, y_{n-1})\in X^n\) such that $x_i\neq x_j$ and $y_i\neq y_j$ for all $i\neq j$, there is \(g\in G\) with \(((x_0)g,\ldots, (x_{n-1})g)=(y_0,\ldots, y_{n-1})\). Furthermore, $G$ is \emph{highly transitive}, if $G$ is $n$-transitive for every $n \in \N$.
\end{definition}

\begin{theorem}\label{maxpoint}
    If \(G\leq \Sym(X)\) is \(2\)-transitive, then for all \(x\in X\) the group \(G_x:=\makeset{g\in G}{\((x)g=x\)}\) is a maximal subgroup of \(G\).
\end{theorem}
\begin{proof}
    Let \(h\in G\backslash G_x\) and let \(H:=\langle h, G_x \rangle\).
    Let \(g\in G\backslash G_x\) be arbitrary, we show that \(g\in H\).
    Let \(a\in G\) be such that \(((x)a, (x)ha)=(x, (x)g)\).
    It follows that \(a\in G_x\) and \((x)g=(x)ha\). Thus \(ga^{-1}h^{-1}\in G_x\) and \(g\in G_xha\subseteq H\).
\end{proof}

\begin{lemma}\label{maxopen}
   Suppose that \(G\leq \Sym(X)\) is highly transitive and has the subspace topology from \(\Sym(X)\). Then \(G\) has exactly \(|X|\) open maximal subgroups.
\end{lemma}
\begin{proof}
      Since $G$ is highly transitive, stabilisers of distinct points are distinct. Thus, \cref{maxpoint} implies that \(G\) has at least \(|X|\) many open maximal subgroups.
    
    Let \(H\) be a open maximal subgroup of \(G\). 
If \(H\) has any finite orbits, it follows that \(H\) is contained in the setwise stabiliser of any of these finite orbits and is thus equal to a setwise stabiliser of a finite set. 
As there are only \(|X|\) finite subsets of \(|X|\) there are only \(|X|\) such setwise stabilisers.
We now complete the proof by showing that every open maximal subgroup of $G$ has a finite orbit. 

Aiming for a contradiction, assume that every orbit of \(H\) is infinite.
As \(1\in H\), there is a finite subset \(F\subseteq X\) such that \(H\) contains the pointwise stabiliser of \(F\) in $G$. Let \(F\) be a minimal (with respect to containment) such set. 
Since $G$ acts highly transitively on $X$, it follows that \(H\) acts highly transitively on \(X\backslash F\).
Let \(x\in F\) be fixed.
As each orbit of \(H\) is infinite, there is an element $h\in H$ with \((x)h\notin F\). Since \(H\) contains the pointwise stabiliser of \(F\), there is also \(k\in H\) which fixes all other elements of \(F\cup (F)h\) and satisfies \(((x)h)k\neq (x)h\). Thus the element \(hkh^{-1}\in H\) fixes all elements of \(F\backslash \{x\}\) and \((x)hkh^{-1}\notin F\).
Thus for all \(y\in X\backslash F\), there is \(f\in H\) pointwise stabilizing \(F\backslash\{x\}\) with \((y)f=x\) and \((x)f=y\).
If \(g\in G\) stabilizes \(F\backslash \{x\}\) pointwise, then there is \(l\in H\) such that \(gl\) lies in the pointwise stabiliser of \(F\). Thus \(g\in H\). So \(H\) contains the pointwise stabiliser of \(F\backslash\{x\}\). This contradicts the choice of \(F\).  
\end{proof}

We may now prove Theorems \ref{maximalsubgroup breakdown} and  \ref{maximal-clopen-open}.
\BBA*
\begin{proof}
    \((1):\) Follows from \cref{maxopen}.
    
    \((2):\) By \cref{closedstab} each finite partition stabiliser is a closed maximal subgroup. It is straightforward to verify that stabilisers of distinct finite clopen partitions are distinct. As there are continuum many clopen subsets of \(\Q\), it follows that there are at least continuum many closed maximal subgroups of \(\Homeo(\Q)\). 
  On the other hand \(\Homeo(\Q)\) is a second-countable group, so it has at most continuum many closed subsets as required.

    \((3):\) Follows from \cref{bigmaxtheorem}.
\end{proof}

\BBB*

\begin{proof}
    Let $\tau$ be a topology on $\Homeo(\Q)$ satisfying the conditions of the theorem. As discussed in the introduction, the results of \cite{Concilio} imply that for all clopen subsets $U$ and $V$ of $\Q$, the set 
    \[\makeset{h\in \operatorname{Homeo}(\Q)}{\((U)h=V\)}\]
    is  open in $\tau$.

    $(1) :$ By \cref{closedstab}, every stabiliser $S$ of a finite partition of $\Q$ into clopen sets is maximal in $\Homeo(\Q)$ and, by the above, $S$ is open. Since there are \(\mathfrak{c}\) such stabilisers $S$, the result follows.
    
     $(2):$ 
     Note that, by the comments above, the stabiliser of any TM-filter is closed in $\tau$. Since $\Q$ is $\aleph_0$-homeoclastic, it now follows from \cref{2chomeoclast} and \cref{theorem-stabilisers-clopen-maximal} that there are  \(2^{\mathfrak{c}}\) maximal subgroups of \(\operatorname{Homeo}(\Q)\) which are closed with respect to $\tau$. There cannot be any more, as there are only \(2^{\mathfrak{c}}\) subsets of \(\operatorname{Homeo}(\Q)\).  
\end{proof}

\end{document}